\documentclass[a4paper]{article}
\usepackage{fullpage}

\usepackage[latin1]{inputenc}
\usepackage[T1]{fontenc}
\usepackage{amsmath,amsthm,amsfonts,amssymb,amsxtra}
\usepackage{array}
\usepackage{enumerate}

\newcommand{\scal}[2]{\langle{#1},{#2}\rangle}
\usepackage{dsfont}
\newcommand{\indicatrice}[1]{\mathds{1}_{{#1}}}

\theoremstyle{plain}                    
\newtheorem{theoreme}{Theorem}
\newtheorem*{theoreme*}{Theorem}
\newtheorem{corollaire}[theoreme]{Corollary}
\newtheorem{lemme}[theoreme]{Lemma}
\newtheorem{proposition}[theoreme]{Proposition}
\newtheorem*{conjectureKLS}{KLS conjecture for convex bodies}
\newtheorem*{conjectureKLSm}{KLS conjecture for log-concave measures}

\theoremstyle{definition} 
\newtheorem*{remarque}{Remark}
\newtheorem*{definition}{Definition}

\title{Spectral gap for some invariant log-concave probability measures}
\author{Nolwen Huet\textsuperscript{1}}
\date{\today}

\begin{document}
\maketitle
\footnotetext[1]{
Institut de Mathématiques de Toulouse,
UMR CNRS 5219,
Université de Toulouse, 
31062 Toulouse, France.
Email: nolwen.huet@math.univ-toulouse.fr.}

\begin{abstract}
We show that the conjecture of Kannan, Lov\'{a}sz, and Simonovits on isoperimetric properties of convex bodies and log-concave measures, is true for log-concave measures of the form $\rho(|x|_B)dx$ on $\mathbb{R}^n$ and $\rho(t,|x|_B) dx$ on $\mathbb{R}^{1+n}$, where $|x|_B$ is the norm associated to any convex body $B$ already satisfying the conjecture. In particular, the conjecture holds  for convex bodies of revolution.
\end{abstract}

\section{Introduction}

Let $K\subset\mathbb{R}^n$ be a convex body. We denote the uniform probability measure on $K$ by $\mu_K$. We say that $K$ is isotropic if 
\begin{itemize}
\item its barycenter $\int x\, d\mu_K(x)$ is 0,
\item $\int \scal{x}{\theta}^2d\mu_K(x)$ is constant over $\theta\in\mathbb{S}^{n-1}$.
\end{itemize}  
This means that the covariance matrix of $\mu_K$ is a multiple of the identity. 
In this case, 
\[
\forall\theta\in\mathbb{S}^{n-1},\quad\int \scal{x}{\theta}^2d\mu_K(x)=\frac{\mathrm{E}_{\mu_K}|X|^2}{n}= \mathrm{E}_{\mu_K}({X_1}^2).
\]
Here $\scal{\,.\,}{\,.\,}$ stands for the euclidean scalar product of $\mathbb{R}^n$ and $|\,.\,|$ the associated  euclidean norm; $\mathbb{S}^{n-1}$ is the euclidean sphere of radius 1, $\mathrm{E}_{\mu}$ is the expectation under the measure $\mu$, and $X_1$ the first coordinate of $X$.
Let us note that others authors require more, namely $\mathrm{Vol}(K)=1$ in \cite{MilPaj} or $\int \scal{x}{\theta}^2d\mu_K=1$ in \cite{KLS}, but we do not.
For such isotropic convex bodies, Kannan, Lov\'{a}sz, and Simonovits state in \cite{KLS} a conjecture (called here KLS conjecture for short) on a certain isoperimetric property,    which can be formulated by results of Maz'ya \cite{mazya2, mazya1} and Cheeger \cite{cheeger}, and also Ledoux \cite{ledoux-geom-bounds}, as follows:
\begin{conjectureKLS}
There exists a universal constant $C$ such that for every $n\ge1$ and every isotropic convex body $K\subset\mathbb{R}^n$,  the Poincaré constant $C_\mathrm{P}(\mu_K)$ of $\mu_K$ is bounded from above by $C\ \mathrm{E}_{\mu_K}({X_1}^2)=C\ {\mathrm{E}_{\mu_K}|X|^2}/n$. 
\end{conjectureKLS}
Let us recall that the Poincaré constant $C_\mathrm{P}(\mu)$ of a measure $\mu$ is the best constant $C$ such that, for every smooth function $f$,
\[
\mathrm{Var}_{\mu} f \le C \int|\nabla f|^2 \,d\mu,
\]
where $\nabla f$ is the gradient of $f$ and $\mathrm{Var}_{\mu} f=\int\big(f-\int f d\mu\big)^2d\mu$ is its variance under $\mu$.
So the conjecture tells us that the Poincaré inequality for convex bodies is tight for linear functions, up to a universal constant.

Up to now, this conjecture is known to be true for $\ell^p$-balls with $p\ge1$ (see \cite{burago-mazya} for the euclidean case $p=2$, \cite{sasha-lp} for $p\in[1,2]$ and \cite{latala-woj} for $p\ge2$
), for the hypercube (see \cite{hadwiger} or \cite{BolLead}), and  for the regular simplex (see \cite{barthe-wolff}). Moreover, as the Poincaré inequality is stable under tensorization (\cite{bobkov-houdre-prod}), an arbitrary product of convex sets satisfying the KLS conjecture, also satisfies it. Quite obviously, it is stable under dilation, too.

One can also extend the definition of isotropy and the conjecture to any log-concave measure $\mu$ on $\mathbb{R}^n$, by just replacing $\mu_K$ by $\mu$ in the statements. Recall that a measure $\mu$
on $\mathbb{R}^n$ is log-concave if 
it satisfies the following Brunn-Minkowski inequality for all compact sets $A$, $B$ and real $\lambda\in[0,1]$:
\[
\mu(\lambda A + (1-\lambda)B)\ge \mu(A)^\lambda \mu(B)^{1-\lambda}.
\]
Equivalently $\mu$ is absolutely continuous with respect to the Lebesgue measure on an affine space of dimension  $m\le n$ and its density's logarithm is a concave function (see \cite{Borell-logconcave}). Then the KLS conjecture becomes
\begin{conjectureKLSm}
Let $\mu$ be a log-concave probability measure on $\mathbb{R}^n$. If $\mu$ is isotropic then $C_\mathrm{P}(\mu)\le C\ {\mathrm{E}_{\mu}({X_1}^2)}=C\ {\mathrm{E}_{\mu}|X|^2}/n$, where $C$ is a universal constant, independent of $n$.
\end{conjectureKLSm}
This conjecture seems rather difficult to tackle. It has been only checked in cases where some additional structure is involved.
For instance, the conjecture holds for any log-concave product measure or more generally for any product of measures satisfying the KLS conjecture \cite{bobkov-houdre-prod}, which generalizes the case of the hypercube. In addition, the conjecture was confirmed by Bobkov  for spherically symmetric log-concave measures in \cite{bobkov-radial-logconcave}, which generalizes the case of the euclidean ball. Let us note also the exitence of dimension-dependent bounds for the Poincaré constant of isotropic log-concave measures:
\begin{equation}\label{borne-small-n}
C_\mathrm{P}(\mu)\le C\ \mathrm{E}_\mu|X|^2\le Cn\, \mathrm{E}_\mu({X_1}^2)
\end{equation}
whenever $\mu$ is an isotropic log-concave measure on $\mathbb{R}^n$. This follows from a theorem proved by Kannan, Lov\'{a}sz, 	and Simonovits \cite{KLS} thanks to their localization lemma, and also deduced by Bobkov \cite{bobkov-logconcave} from its isoperimetric inequality for log-concave measures.
\begin{theoreme}[Kannan--Lov\'{a}sz--Simonovits \cite{KLS}, Bobkov \cite{bobkov-logconcave}]\label{thm:Pvar}
Let $X=(X_1,\ldots,X_n)$ be a log-concave random vector in $\mathbb{R}^n$ of law $\mu$. Then, there exists a universal constant $C$ such that
\[
C_\mathrm{P}(\mu)\le C\, \mathrm{E}_\mu\big|X-\mathrm{E}(X)\big|^2 = C\sum_{i=1}^n\mathrm{Var}_\mu(X_i).
\]
\end{theoreme}
This was improved by Bobkov \cite{bobkov-isop-const} to:
\[
C_\mathrm{P}(\mu)\le C\, \left(\mathrm{Var}|X|^2\right)^{1/2}.
\]
Using  Klartag's power-law estimates \cite{klartag-TCL-powerlaw} for $\mathrm{Var}|X|^2$, this yields (cf \cite{EMilmanConvIsopSpectConc}) the following improvement of  the bound \eqref{borne-small-n}: for any $\varepsilon>0$, there exists $C_\varepsilon>0$ such that
\begin{equation*}
C_\mathrm{P}(\mu)\le  C_\varepsilon n^{1-1/5+\varepsilon}\ {\mathrm{E}({X_1}^2)}.
\end{equation*}
If moreover $\mu$ is uniform on an unconditional convex set, i.e. invariant under coordinate reflections, Klartag shows \cite{klartag-uncond} that
\begin{equation*}\label{borne-incond-klartag}
C_\mathrm{P}(\mu)\le  C (\log n)^{2}\ {\mathrm{E}({X_1}^2)}.
\end{equation*}

In Section \ref{sec:B-sym}, we deal with the case of log-concave measures symmetric  with respect to the norm associated to a convex body $B$ satisfying the KLS conjecture. It encompasses the result of Bobkov on spherically symmetric log-concave measures. Bobkov's proof relied on the tensorization of the radial measure with the uniform measure on the sphere. He used the following property of the law of the radius: 
\begin{theoreme}[Bobkov \cite{bobkov-radial-logconcave}]\label{prop:rad}
Let $\nu$ be a probability measure on $\mathbb{R}_+$ defined by
\[
\nu(dr)=r^{n-1}\rho(r)\indicatrice{\mathbb{R}_+}(r)dr
\]
with $\rho$ log-concave. Then 
\[
\mathrm{Var}_\nu(r)\le\frac{\mathrm{E}_\nu(r^2)}{n}.
\]
\end{theoreme}
Here, a natural idea would be to use the polar representation $X=R\theta$, where $R=|X|_B$ is the norm of $X$, and the distribution of $\theta$ is the cone measure on $\partial B$. However, in the non-euclidean case, the differentiation on $\partial B$ is more difficult to handle. So we choose to  decompose $X$ into $SU$ where $S\in\mathbb{R}_+$ and $U$ is uniform on $B$.
Using the same method, we study  convex bodies of revolution in Section \ref{sec:B-revol}, as well as more general log-concave measures of the form $\rho(x_0,|x_1|_{B_1},\ldots,|x_k|_{B_k})dx$, with convex bodies $B_i$ satisfying the KLS conjecture.
To simplify the statements, we use the following definition. 
\begin{definition}
An isotropic convex body $B$ (respectively an isotropic log-concave probability $\mu$) 
is said to satisfy KLS with constant $C$ if $C_\mathrm{P}(\mu_B)\le C \mathrm{E}_{\mu_B}(|X|^2)/n$ (respectively $C_\mathrm{P}(\mu)\le C \mathrm{E}_{\mu}(|X|^2)/n$). 
\end{definition}
If $X$,$Y$ are random variables and $\mu$ is a probability measure, we will note $X\sim Y$ when the two random variables have the same distribution, and $X\sim \mu$ when $\mu$ is the law of  $X$.
We close this introduction by stating a very useful theorem due to E. Milman (see  \cite[Theorem 2.4]{EMilmanConvIsopSpectConc} where  a stronger result  is stated).
\begin{theoreme}[E. Milman \cite{EMilmanConvIsopSpectConc}]
\label{thm:Pinfty}
Let $\mu$ be a log-concave probability measure. If there exists $C>0$ such that for every smooth function $f$, 
\[
\mathrm{Var}_\mu(f)\le C \big\||\nabla f|\big\|_\infty^2,
\]
then $C_\mathrm{P}(\mu)\le c\, C$, where $c>0$ is a universal constant.
\end{theoreme}

\section{$B$-symmetric log-concave measures}\label{sec:B-sym}

Let $B$ be a convex body on $\mathbb{R}^n$, whose interior contains 0. We associate to $B$ the (non-necessary symmetric) norm $|\,.\,|_B$ defined by
\[
|x|_B=\inf\left\{\lambda>0,\ \frac x \lambda\in B\right\}.
\] 
Then $B$ and its boundary $\partial B$ correspond respectively to the unit ball and the unit sphere for this norm. Let us note respectively  $\mu_B$ and $\sigma_B$ the uniform measure on $B$ and the cone measure on $\partial B$ normalized so as to be probability measures.
Recall that the cone measure is characterized by  the following decomposition formula for all measurable $f:\mathbb{R}^n\to\mathbb{R}$:
\[
\int_B f \,d\mu_B = \int_0^1  nr^{n-1} \int_{\partial B} f(r\theta) 
\,d\sigma_B(\theta)\ dr.
\]
Let $\mu$ be a probability measure on $\mathbb{R}^n$ with density  $\rho(|x|_B)$ with respect to the Lebesgue measure. If $\rho$ is log-concave and non-increasing then $\mu$ is log-concave. In that case, we say that $\mu$ is a $B$-symmetric log-concave probability measure. Note that the isotropy of $\mu$ amounts to the isotropy of $B$.

\begin{proposition}\label{prop:B-sym}
Let $B$ be an isotropic convex body of $\mathbb{R}^n$ satisfying KLS with constant $C$. 
Then any probability measure $\mu(dx)=\rho(|x|_B)dx$ with $\rho$ log-concave and non-increasing on $\mathbb{R}^+$, satisfies KLS with constant $\alpha C$, where $\alpha$ is a universal constant.
\end{proposition}

To prove the latter, we use the following decomposition of $\mu$.
\begin{lemme}\label{lem:decomposition}
Let $X$ be a random variable on $\mathbb{R}^n$ of law $\mu(dx)=\rho(|x|_B)dx$ with $\rho$ log-concave and non-increasing on $\mathbb{R}^+$. Then
$X\sim SU$ where $U$ and $S$ are independent random variables respectively on $B$ and $\mathbb{R}^+$,  of law $U\sim\mu_B$ and $S\sim - \mathrm{Vol}(B) s^{n}\rho'(s) \indicatrice{R^+}(s) ds$.
\end{lemme}
Note that, as $\rho$ is log-concave, $\rho$ is locally Lipschitz, and thereby is  differentiable almost everywhere and satisfies the fundamental theorem of calculus (see for instance \cite{rudin}).
\begin{proof}
Let us recall the classical polar decomposition: if $\theta$ and $R$ are independent random variables respectively on $\partial B$ and $\mathbb{R}^+$,  of law $\theta\sim\sigma_B$ and $R\sim n \mathrm{Vol}(B) r^{n-1}\rho(r)\indicatrice{R^+}(r) dr$, then $R\theta\sim \mu$. 
Let $T\sim n  t^{n-1} \indicatrice{[0,1]}(t) dt$ be independent of $\theta\sim\sigma_B$. By the same polar decomposition, $U=T\theta$ is uniform on $B$. So, if $S\sim - \mathrm{Vol}(B) s^{n}\rho'(s) \indicatrice{R^+}(s) ds$ is independent of $T$ and $\theta$,  it remains only to show that $ST\sim R$ to prove that $SU=(ST)\theta\sim\mu$. 
 Let $f:\mathbb{R}^n\to \mathbb{R}$ be a measurable function, then
\begin{align*}
\mathrm{E}\big(f(ST)\big)&= \iint f(st)\ n\mathrm{Vol}(B)\, t^{n-1}s^n \big(-\rho'(s)\big)\,   \indicatrice{[0,1]}(t) \indicatrice{R^+}(s)\, dtds\\
&= \int f(r)\ n\mathrm{Vol}(B)\, r^{n-1} \indicatrice{R^+}(r) \left(\int_r^{+\infty} - \rho'(s) ds\right) dr\\
&= \int f(r)\ n \mathrm{Vol}(B) r^{n-1}\rho(r)\indicatrice{R^+}(r)dr.
\end{align*}
In the last equality, we use that $-\rho'(x)\ge 0$ and $\rho(x)\to 0$ when $x\to+\infty$. Actually, as $\rho$ is log-concave, non-increasing, and non-constant as a density,   there exists  c>0 such that $\rho(x)\le e^{-cx}$ for $x$ large enough. 
\end{proof}
We  need then a kind of Poincaré inequality for $S$ in the above representation.

\begin{lemme}\label{lem:PS}
Let $S$ be a random variable on $\mathbb{R}^+$ of law $\eta(ds)=-s^n\rho'(s)ds$ where $\rho:\mathbb{R}^+\to\mathbb{R}^+$ is a log-concave non-increasing function. Then, for every smooth function $f$ on $\mathbb{R}^+$
\[
\mathrm{Var}_\eta(f)\le c \frac{\mathrm{E}(S^2)}n \| f'\|_\infty^2, 
\]
where $c>0$ is a universal constant.
\end{lemme}
\begin{proof}
By integration by parts, we see that $ -\int_0^{+\infty}  r^n \rho'(r) dr =\int_0^{+\infty} n  r^{n-1} \rho(r) dr=1$. So $\nu(dr)=n  r^{n-1} \rho(r) dr$ is a probability measure, and by Theorems \ref{thm:Pvar} and \ref{prop:rad}, if $R\sim \nu$, then there exists a universal constant $c>0$ such that
\[
C_\mathrm{P}(\nu)\le c\mathrm{Var} (R) \le c\frac{\mathrm{E} (R^2)}{n}.
\]
\begin{remarque}
As pointed out by the referee, Theorem \ref{thm:Pvar} is not really necessary here to derive the first inequality since we are in one dimension. For instance, the inequality is true with $c=12$ by Corollary 4.3 of \cite{bobkov-logconcave}.
\end{remarque} 
Let $f$ be a smooth function on $\mathbb{R}^+$ such that $\mathrm{E}\big(f(R)\big)=0$ and $\| f'\|_\infty<\infty$. We perform again an integration by parts:
\begin{align*}
\mathrm{E}\left(f^2(S)\right)&= - \int_0^{+\infty} f^2(r) r^n \rho'(r) dr \\
&=\int_0^{+\infty} \big(f^2(r) n  r^{n-1} + 2f(r)f'(r) r^n\big)\rho(r) dr\\
&=\mathrm{E}\big(f^2(R)\big) + \frac2n \mathrm{E}\big(Rf(R)f'(R)\big).
\end{align*}
In the second equality, we use that $f^2(r)r^n\rho(r)\to 0$ when $r\to +\infty$ since $f$ was assumed Lipschitz and since $\rho$ decays exponentially.
The Poincaré inequality for $R$ leads to 
\[
\mathrm{E}\big(f^2(R)\big)\le c \frac{\mathrm{E} (R^2)}{n}\| f'\|_\infty^2.
\]
Moreover,
\begin{align*}
\mathrm{E}\big(Rf(R)f'(R)\big)&\le \| f'\|_\infty \sqrt{\mathrm{E}(R^2)}\sqrt{ \mathrm{E}\big(f^2(R)\big)}\\
&\le \| f'\|_\infty \sqrt{\mathrm{E}(R^2)} \sqrt{c \frac{\mathrm{E} (R^2)}{n}\| f'\|_\infty^2}\\
&=\sqrt{\frac cn}\mathrm{E} (R^2) \| f'\|_\infty^2.
\end{align*}
Thus,
\[
\mathrm{E}\left(f^2(S)\right) \le
\left(c+2\sqrt{c/n}\right) \frac{\mathrm{E}(R^2)}{n} \| f'\|_\infty^2.
\]
Now, for any smooth function $g$ on $\mathbb{R}^+$, we set $f=g-\mathrm{E}\big(g(R)\big)$. As  $\mathrm{Var}_\eta (g)\le \mathrm{E}\big(g(S)-a\big)^2$ for every $a\in\mathbb{R}$, it holds
\[
\mathrm{Var}_\eta(g) \le
\left(c+2\sqrt{c/n}\right) \frac{\mathrm{E}(R^2)}{n} \| g'\|_\infty^2.
\]
To conclude, let us remark that

\[
\mathrm{E}(R^2)
=\int_0^{+\infty} n  r^{n+1} \rho(r) dr
=-\frac{n}{n+2}\int_0^{+\infty}  r^{n+2} \rho'(r) dr
= \frac{n}{n+2}\mathrm{E}(S^2).
\]
\end{proof}

We can now prove  Proposition \ref{prop:B-sym} by tensorization.
\begin{proof}[Proof of Proposition \ref{prop:B-sym}]
Let $\mu(dx)=\rho(|x|_B)dx$ be a probability measure with $\rho$ log-concave non-increasing. Let  $S$ and $U$ be independent random variables respectively on $\mathbb{R}^+$ and uniform on $B$, such that $X=SU\sim \mu$ as in Lemma \ref{lem:decomposition}. 
Let $f$ be a smooth function on $\mathbb{R}^n$. By Lemma \ref{lem:PS}, there exists a universal constant $c>0$ such that 
\begin{align}\label{intnu}
\mathrm{E}_S\big(f^2(SU)\big) &\le \left[\mathrm{E}_S\big(f(SU)\big)\right]^2 + c\frac{\mathrm{E}(S^2)}{n} \max_{s\ge0}\left(\scal{\nabla f(sU)}{U}^2\right)\notag\\
&\le \left[\mathrm{E}_S\big(f(SU)\big)\right]^2 
+ c\frac{\mathrm{E}(S^2)}{n} |U|^2 \big\||\nabla f|\big\|_\infty^2.
\end{align}
As $B$ satisfies KLS with constant $C$, we can also apply the Poincaré inequality to  $u\mapsto\mathrm{E}_S\big(f(Su)\big)$:
\begin{align*}
\mathrm{E}_U\left[\mathrm{E}_S\big(f(SU)\big)\right]^2&\le 
\left[\mathrm{E}_U\mathrm{E}_S\big(f(SU)\big)\right]^2 + C\frac{\mathrm{E}(|U|^2)}n 
\mathrm{E}_U\left|\mathrm{E}_S \big( S\nabla f(SU)\big)\right|^2\\
&\le \left[\mathrm{E}_U\mathrm{E}_S\big(f(SU)\big)\right]^2 + C\frac{\mathrm{E}(|U|^2)}n \mathrm{E} (S^2)\big\||\nabla f|\big\|_\infty^2.
\end{align*}
So, if we take the expectation with respect to $U$ in \eqref{intnu}, we obtain:
\begin{align*}
\mathrm{Var}_\mu f &\le (c+C) \frac{\mathrm{E}(S^2)\mathrm{E}(|U|^2)}{n} \big\||\nabla f|\big\|_\infty^2\\
&=
(c+C) \frac{\mathrm{E}(|X|^2)}{n} \big\||\nabla f|\big\|_\infty^2,
\end{align*}
since $S$ and $U$ are independent, and $X=SU$. Let us remark that $C\ge1$ as it can be seen by testing  the Poincaré inequality on linear functions, so that $c+C\le (1+c)C$. We conclude by Theorem \ref{thm:Pinfty}.
\end{proof}

\section{Convex bodies of revolution}\label{sec:B-revol}
The same method works with convex bodies of revolution or more generally for convex bodies $K\subset\mathbb{R}^{n+1}$ defined by 
\[
K=\{(t,x)\in I\times\mathbb{R}^{n}; |x|_B\le R(t) \},
\]
where $I$ is a bounded interval of $\mathbb{R}$ and $R:I\to\mathbb{R}^+$ is a concave function.  Actually we show the KLS conjecture for corresponding measures $\rho(t,|x|_B)dtdx$ and more generally  log-concave measures of the form $\rho(x_0,|x_1|_{B_1},\ldots,|x_k|_{B_k})dx$.

Let $n_0$, \ldots, $n_k$ be positive integers and $N=\sum n_i$. 
Let $B_1,\ldots,B_k$ be isotropic convex bodies respectively of $\mathbb{R}^{n_1},\ldots, \mathbb{R}^{n_k}$ satisfying KLS with constant $C$. Let $\rho:\mathbb{R}^{n_0}\times(\mathbb{R}^+)^k\to\mathbb{R}^+$ be a smooth log-concave function such that:
\begin{enumerate}[(i)]
\item $\partial_j \rho \le0$ for all $j=1,\ldots,k$, and
\item \label{condbiz}$(-1)^j\partial^j_{k-(j-1),\ldots,k}\rho\ge0$ for all $j=1,\ldots,k$,
\end{enumerate}
where  $\partial_i\rho$ is the partial derivative of $s_i\mapsto\rho(x_0,s_1,\ldots,s_k)$. 

\begin{proposition}\label{revol}
With the above notations and hypotheses, if $\mu(dx)=\rho(x_0,|x_1|_{B_1},\ldots,|x_k|_{B_k})dx$ defines an isotropic probability measure on $\mathbb{R}^N$, then $\mu$ satisfies KLS with a constant depending only on $C$, $n_0$, and $k$.
\end{proposition}
The above condition (\ref{condbiz}) on $\rho$ is not satisfactory, but we did not succed in getting rid of it with our method of decomposition. When using the classical polar decomposition instead, we do not need anymore this condition, but an other one arises on the $B_i$'s themselves, because of the differentiation on $\partial B_i$. Nevertheless, when $n_0=k=1$ and $B_1$ is an euclidean ball, the latter proposition is sufficient to deduce the following.
\begin{corollaire}
The convex bodies of revolution satisfy the KLS conjecture.
\end{corollaire}
\begin{proof}[Proof of Proposition \ref{revol}]
The same method as for $B$-symmetric measures applies.
In the same way as in Lemma \ref{lem:decomposition}, let $(X_0,S)=(X_0,S_1,\ldots,S_k)$ and $U_1$, \ldots, $U_k$ be independent variables respectively on $\mathbb{R}\times\left(\mathbb{R}^+\right)^k$ and $B_1$, \dots, $B_k$, with
\[
(X_0,S)\sim(-1)^k \mathrm{Vol}(B_1)\cdots\mathrm{Vol}(B_k)\ 
s_1^{n_1}\cdots s_k^{n_k}
\partial^k_{1,\ldots,k} \rho(x_0,s)\ \indicatrice{\left(\mathbb{R}^+\right)^k}(s)dx_0ds,
\]
\[
\forall i, \quad U_i\sim \mu_{B_i}.
\]
Then $X=(X_0,S_1U_1,\ldots,S_kU_k)$ is of law $\mu$.
Moreover, we can suppose without loss of generality that that for all $i$,
\begin{equation*}
\mathrm{E}|U_i|^2=1. 
\end{equation*} 
Else, we choose $\lambda_i^{-2}=\mathrm{E}(|U_i|^2)$ such that $V_i\sim\mu_{\lambda_i B}$ satisfies $\mathrm{E}(|{V_i}|^2)=1$, and replace $\rho$ by $\rho_\lambda:(x_0,s_1,\ldots,s_k)\mapsto\rho(x_0,\lambda_1 s_1,\ldots,\lambda_k,s_k)$ which is still log-concave and non-increasing with respect to its $k$ last variables. In this case, $\mu(dx)=\rho_\lambda(x_0,|x_1|_{\lambda_1 B_1},\ldots,|x_k|_{\lambda_k B_k})dx$. 

Recall now that $\mu$ is isotropic, i.e. that 
\[
\mathrm{E}(X)=0
\]
and 
\[
\forall y\in\mathbb{R}^N,\quad
\mathrm{E}\left(\scal{X}{y}^2\right)
=\frac{\mathrm{E}(|X|^2)}{N}|y|^2.
\]
This implies that each $B_i$ had to be itself an isotropic convex body, $X_0$ is an isotropic log-concave variable, and \begin{equation}\label{eq:SUiso}
\forall i,\quad \frac{\mathrm{E}(S_i^2)}{n_i}=\frac{\mathrm{E}(|X_0|^2)}{n_0}=\frac{\mathrm{E}(|X|^2)}{N},
\end{equation}
since $\mathrm{E}(|U_i|^2)=1$.
In the same spirit as Lemma \ref{lem:PS}, one can show 
\begin{lemme}\label{lem:PSk}
For any smooth function $f$ on $\mathbb{R}^{n_0}\times(\mathbb{R}^+)^k$, 
\[
\mathrm{Var}_{(X_0,S)}(f)\le c(n_0+k^2)\frac{\mathrm{E}|X|^2}N \big\||\nabla f|\big\|_\infty^2 ,\] 
where $c>0$ is a universal constant. 
\end{lemme} 
\noindent 
Let us postpone the proof of this lemma and show how to deduce the claim of the proposition. Let $f$ be a smooth function on $\mathbb{R}^{N}$. Let us denote  by $\nabla_if$ the derivative of $$x_i\in\mathbb{R}^{n_i}\mapsto f(x_0,x_1,\ldots,x_k),$$ and
$su=(s_1u_1,\ldots,s_ku_k)$ whenever $s=(s_1,\ldots,s_k)\in\mathbb{R}^k$ and $u=(u_1,\ldots,u_k)$ with $u_i\in\mathbb{R}^{n_i}$.
As in the proof of Proposition \ref{prop:B-sym}, we apply Poincaré inequality for $(X_0,S)$ and then for each $U_i$. First we consider $(x_0,s)\mapsto f(x_0,sU)$, where $U=(U_1,\ldots,U_k)$: 
\begin{multline*}
\mathrm{E}_{(X_0,S)}\big(f^2(X_0,SU)\big)
\le \left[\mathrm{E}_{(X_0,S)}\big(f(X_0,SU)\big) \right]^2 \\
{+ c(n_0+k^2)\frac{\mathrm{E}(|X|^2)}{N} \max_{(x_0,s)\in\mathbb{R}\times(\mathbb{R}^+)^k}
\left(|\nabla_0 f(x_0,sU)|^2 +\sum_{i=1}^k \scal{\nabla_i f(x_0,sU)}{U_i}^2 \right) }
\\
\le \left[\mathrm{E}_{(X_0,S)}\big(f(X_0,SU)\big)\right]^2 + c(n_0+k^2)\frac{\mathrm{E}(|X|^2)}{N} \left(1+{\textstyle\sum_i|U_i|^2}\right) \big\||\nabla f|\big\|_\infty^2.
\end{multline*}
Then, as each $B_i$ satisfies KLS with constant $C$, \begin{align*}
\mathrm{E}_U\left[\mathrm{E}_{(X_0,S)}\big(f(X_0,SU)\big)\right]^2 &\le 
\left[\mathrm{E}\big(f(X)\big)\right]^2 +\sum_{i=1}^k C\frac{\mathrm{E}(|U_i|^2)}{n_i} 
\mathrm{E}_{U_i}\left|\mathrm{E}_{(X_0,S)} \big( S_i\nabla_i f(X_0,SU)\big)\right|^2\\
&\le \left[\mathrm{E}\big(f(X)\big)\right]^2 + \sum_{i=1}^k C\frac{\mathrm{E}(|U_i|^2)}{n_i} 
\mathrm{E}(S_i^2)\mathrm{E}(|\nabla_i f(X)|^2)\\
&\le \left[\mathrm{E}\big(f(X)\big)\right]^2 + C\frac{\mathrm{E}(|X|^2)}{N} 
\mathrm{E}(|\nabla f(X)|^2).
\end{align*}
The second inequality comes from Cauchy--Schwarz inequality and the independence of $S$ and $U$.
It follows that
\begin{align*}
\mathrm{Var}_\mu(f)&\le \left[ c(n_0+k^2) \big(1+{\textstyle\sum_i\mathrm{E}(|U_i|^2)}\big) 
+ C\right]\frac{\mathrm{E}(|X|^2)}{N}
\big\||\nabla f|\big\|_\infty^2\\
&=\left[2c(n_0+k^2)(1+k)+C\right]\frac{\mathrm{E}(|X|^2)}{N}\big\||\nabla f|\big\|_\infty^2.
\end{align*}
 We are done, thanks to Theorem \ref{thm:Pinfty}.
\end{proof}
\begin{proof}[Proof of Lemma \ref{lem:PSk}]
Let $R=(R_1,\ldots,R_k)$ be a random variable on $(\mathbb{R}^+)^k$ such that 
\[
(X_0,R)\sim \nu(dx_0dr)=\rho(x_0,r)
\left(\prod_{i=1}^k\mathrm{Vol}(B_i)n_i \indicatrice{\mathbb{R}^+}(r_i) r_i^{n_i-1}\,dr_i\right)  \,dx_0.
\]
Then one can see thanks to the classical polar decomposition of $\mu$ that 
\[
\forall i,\quad R_i\sim S_i|U_i|_{B_i}.
\]
In particular, and thanks to \eqref{eq:SUiso},
\[
\forall i,\quad \frac{\mathrm{E}(R_i^2)}{n_i}=\frac{n_i}{n_i+2}\frac{\mathrm{E}(S_i^2)}{n_i}=\frac{n_i}{n_i+2}\frac{\mathrm{E}(|X|^2)}{N}\le\frac{\mathrm{E}(|X|^2)}{N}.
\]
Moreover $\nu$ is a  log-concave (non-isotropic) measure  and by Theorem \ref{thm:Pvar}, there exists a universal constant $c>0$ such that
\[C_P(\nu) \le c\left(\mathrm{E}(|X_0|^2)+\sum_i\mathrm{Var}(R_i)\right).\] 
Now  the density of $R_i$ is proportional to 
\[
r_i^{n_i-1}\int \rho(x_0,r) dx_0dr_1\cdots dr_{i-1}dr_{i+1}\cdots dr_k
\] 
and  $r_i\mapsto\int \rho(x_0,r) dx_0dr_1\cdots dr_{i-1}dr_{i+1}\cdots dr_k$ is a non-increasing function which is also log-concave by the Prékopa--Leindler theorem \cite{pre1,lei,pre2}. According to Theorem \ref{prop:rad}, 
\[
\mathrm{Var} (R_i)\le \frac{\mathrm{E}(R_i^2)}{n_i}\le\frac{\mathrm{E}(|X|^2)}{N}.
\] 
Using \eqref{eq:SUiso}, we conclude that:
\begin{equation}\label{eq:Poincare-nu}
C_P(\nu) \le c(n_0+k)\frac{\mathrm{E}(|X|^2)}{N}.
\end{equation}

As in the proof of Lemma \ref{lem:PS}, to prove the claim, it is enough to consider smooth functions $f$ on $\mathbb{R}^{n_0}\times(\mathbb{R}^+)^k$ such that $\mathrm{E}\big(f(R)\big)=0$ and $\big\||\nabla f|\big\|_\infty<\infty$. We can also assume $f\ge0$, else we consider $g=|f|$ instead of $f$: $g$ is differentiable almost everywhere and $\big\||\nabla g|\big\|_\infty=\big\||\nabla f|\big\|_\infty$. We set $Z=\mathrm{Vol}(B_1)\cdots\mathrm{Vol}(B_k)$. Then we apply  $k$ successive integrations by parts and bound the derivatives of $f$  by $\big\||\nabla f|\big\|_\infty$ each times they appear: 
\begin{align}
&\mathrm{E}
{\big(f^2(X_0,S)\big)} \notag\\
&=Z\int f^2(x_0,r) \ (-1)^k  
r_1^{n_1}\cdots r_k^{n_k}
\partial^k_{1,\ldots,k} \rho(x_0,r)\ \indicatrice{\left(\mathbb{R}^+\right)^k}(r)dx_0dr \notag\\
&=Z\int \left[f^2  + 2\frac{r_1}{n_1}f\partial_1f \right]
(-1)^{k-1}  
n_1r_1^{n_1-1}r_2^{n_2}\cdots r_k^{n_k}
\partial^{k-1}_{2,\ldots,k} \rho(x_0,r)\ \indicatrice{\left(\mathbb{R}^+\right)^k}(r)dx_0dr \notag\\ 
&\le Z\int \left[f^2  + 2\frac{r_1}{n_1}f\big\||\nabla f|\big\|_\infty \right]
(-1)^{k-1}  
n_1r_1^{n_1-1}r_2^{n_2}\cdots r_k^{n_k}
\partial^{k-1}_{2,\ldots,k} \rho(x_0,r)\ \indicatrice{\left(\mathbb{R}^+\right)^k}(r)dx_0dr \notag\\ 
&= Z\int \left[f^2 + 2\frac{r_1}{n_1}f\big\||\nabla f|\big\|_\infty + 2 \frac{r_2}{n_2}f\partial_2f + 2\frac{r_1r_2}{n_1n_2}\partial_2f\big\||\nabla f|\big\|_\infty \right]
\notag\\
&\hspace{8em}
(-1)^{k-2}  
n_1n_2r_1^{n_1-1}r_2^{n_2-1}r_3^{n_3}\cdots r_k^{n_k}
\partial^{k-2}_{3,\ldots,k} \rho(x_0,r)\ \indicatrice{\left(\mathbb{R}^+\right)^k}(r)dx_0dr\notag\\ 
&\le Z\int \left[f^2  + 2 \left(\frac{r_1}{n_1}+\frac{r_2}{n_2}\right)f\big\||\nabla f|\big\|_\infty + 2\frac{r_1r_2}{n_1n_2}\big\||\nabla f|\big\|_\infty^2 \right]
\notag\\
&\hspace{8em}
(-1)^{k-2}  
n_1n_2r_1^{n_1-1}r_2^{n_2-1}r_3^{n_3}\cdots r_k^{n_k}
\partial^{k-2}_{3,\ldots,k} \rho(x_0,r)\ \indicatrice{\left(\mathbb{R}^+\right)^k}(r)dx_0dr\notag\\ 
&\vdots\notag\\
&\le \mathrm{E}\big(f^2(X_0,R)\big) 
+ 2 \big\||\nabla f|\big\|_\infty \mathrm{E}\left( f(X_0,R) \sum_i \frac{R_i}{n_i}\right)
+ 2 \big\||\nabla f|\big\|_\infty^2 \mathrm{E}\left(\sum_{i<j}\frac{R_iR_j}{n_in_j}\right).\label{eq:intpartk}
\end{align}
By Cauchy--Schwarz inequality,
\begin{equation}\label{eq:CS1k}
\mathrm{E}\left( f(X_0,R) \sum_i \frac{R_i}{n_i}\right)
\le \sqrt{\mathrm{E}\big(f^2(X_0,R)\big)}\sqrt{\mathrm{E}\left(\sum_i\frac {R_i}{n_i}\right)^2},
\end{equation}
and
\begin{equation}\label{eq:CS2k}
\mathrm{E}\left(\sum_i\frac {R_i}{n_i}\right)^2
\le \left(\sum_i\frac1{n_i}\right)\mathrm{E}\left(\sum_i\frac {R_i^2}{n_i}\right)
\le k^2 \frac{\mathrm{E}|X|^2}N.
\end{equation}
The Poincaré inequality \eqref{eq:Poincare-nu} satisfied by $\nu$ leads to 
\begin{equation}\label{eq:PRk}
\mathrm{E}\big(f^2(X_0,R)\big)\le c(n_0+k)\frac{\mathrm{E}(|X|^2)}{N}\big\||\nabla f|\big\|_\infty^2.
\end{equation}
We note that $\sum_{i<j}\frac{R_iR_j}{n_in_j}\le \left(\sum_i \frac{R_i}{n_i}\right)^2$ and plug the estimates \eqref{eq:CS1k}, \eqref{eq:CS2k}, and \eqref{eq:PRk} in the inequality \eqref{eq:intpartk}, so that
\begin{align*}
\mathrm{E}\big(f^2(X_0,S)\big)&\le
\left(c(n_0+k)+2\sqrt{c(n_0+k)}k+2k^2\right) \frac{\mathrm{E}|X|^2}N \big\||\nabla f|\big\|_\infty^2\\
&\le
\left(2c(n_0+k)+3k^2\right) \frac{\mathrm{E}|X|^2}N \big\||\nabla f|\big\|_\infty^2.
\end{align*}
Consequently, the assertion follows.
\end{proof}

\begin{remarque}
In the special case of the euclidean ball, one can improve the above estimate  using the classical polar decomposition $X=(R_1\theta_1,\ldots,R_k\theta_k)$ with $\theta_i\sim \sigma_{\mathbb{S}^{n_i-1}}$. Moreover, we can drop the condition (\ref{condbiz}) of page \pageref{condbiz}. We omit the details and only state:
\begin{proposition}
Let $\mu(dx)=\rho(x_0,|x_1|,\ldots,|x_k|)dx$ be an isotropic log-concave probability measure on $\mathbb{R}^N$, with $x_i\in\mathbb{R}^{n_i}$. Then there exists a universal constant $c$ such that 
\[
C_\mathrm{P}(\mu) \le c (n_0+k) \frac{\mathrm{E}_\mu|X|^2}N.
\]
\end{proposition}\end{remarque}

\paragraph{Acknowledgments.} I would like to thank Franck Barthe for fruitful discussions, and the referee for his judicious remarks and helpful suggestions.

\end{document}